\documentclass[noinfoline]{imsart}
\RequirePackage[OT1]{fontenc}
\RequirePackage{amsthm,amsmath}
\RequirePackage[numbers]{natbib}
\RequirePackage{hypernat}
\usepackage{amssymb,graphicx,url,bbm}
\usepackage[colorlinks,citecolor=blue,urlcolor=blue]{hyperref}

\startlocaldefs
\def\d{\delta}

\def\argmax{\mbox{argmax}}

\def\ai{\text{\rm Ai}}
\def\bi{\text{\rm Bi}}
\def\hi{\text{\rm Hi}}

\def\R{\mathbb R}

\def\d{\delta}

\def\M{{\cal M}}

\def\l{\lambda}
\def\labda1{\lambda_1}
\def\labda2{\lambda_2}

\def\e{\varepsilon}
\def\f{\phi}
\def\t{\tau}

\def\s{\sigma}

\def\comment#1{\relax}

\def\=in{\mathop{\rm =}}

\newtheorem{theorem}{Theorem}[section]
\newtheorem{corollary}{Corollary}[section]
\newtheorem{lemma}{Lemma}[section]

\newtheorem{remark}{Remark}[section]

\numberwithin{equation}{section}
\theoremstyle{plain}
\endlocaldefs

\usepackage{amsmath,amssymb,here,amsfonts,latexsym,sym1,graphicx,here}
\usepackage{amsthm,curves}
\usepackage{natbib}
\usepackage{graphicx}
\usepackage{caption}
\usepackage{float}

\begin{document}

\begin{frontmatter}
\title{Chernoff's distribution and differential equations of parabolic and Airy type}
\runtitle{Chernoff's distribution}

\begin{aug}
\author{\fnms{Piet} \snm{Groeneboom}\corref{}\ead[label=e1]{P.Groeneboom@tudelft.nl}}
\ead[label=u1,url]{http://dutiosc.twi.tudelft.nl/\textasciitilde pietg/}
\address{Delft University\\
\printead{e1},\printead{u1}}
\author{\fnms{Steven} \snm{Lalley}\ead[label=e2]{lalley@galton.uchicago.edu}}
\ead[label=u2,url]{http://www.stat.uchicago.edu/faculty/lalley.shtml}
\address{University of Chicago\\
\printead{e2},\printead{u2}}
\author{\fnms{Nico} \snm{Temme}\ead[label=e3]{Nico.Temme@cwi.nl}}
\ead[label=u3,url]{http://homepages.cwi.nl/\textasciitilde nicot/}
\address{CWI\\
\printead{e3},\printead{u3}}
\affiliation{Delft University, CWI and University of Chicago}
\runauthor{P.\ Groeneboom, N.M.\ Temme and S.P.\ Lalley}
\end{aug}

\begin{abstract}
We give a direct derivation of the distribution of the maximum and the location of the maximum of one-sided and two-sided Brownian motion with a negative parabolic drift. The argument  uses a relation between integrals of special functions, in particular involving integrals with respect to functions which can be called ``incomplete Scorer functions". The relation is proved by showing that both integrals, as a function of two parameters, satisfy the same extended heat equation, and the maximum principle is used to show that these solution must therefore have the stated relation. Once this relation is established, a direct derivation of the distribution of the maximum and location of the maximum of Brownian motion minus a parabola is possible, leading to a considerable shortening of the original proofs.
\end{abstract}



\end{frontmatter}

\section{Introduction}
\label{section:intro}
\setcounter{equation}{0}
Let $\{W(t):\,t\in\R\}$ be standard two-sided Brownian motion, originating from zero. The determination of the distribution of the (almost surely unique) location of the maximum of $\{W(t)-t^2:t\in\R\}$ has a long history, which probably started with Chernoff's paper \cite{chernoff:64} in a study of the limit distribution of an estimator of the mode of a distribution. In the latter paper the density of the location of the maximum of $\{W(t)-t^2:t\in\R\}$, which we will denote by
\begin{equation}
\label{def_Z}
Z=\argmax_t\{W(t)-t^2,\,t\in\R\},
\end{equation}
is characterized in the following way. Let $u(t,x)$  be the solution of the heat equation
$$
\frac{\partial}{\partial t}u(t,x)=-\tfrac12\frac{\partial^2}{\partial x^2}u(t,x),
$$
for $x\le t^2$, under the boundary conditions
$$
u(t,x)\ge0,\qquad u(t,t^2)\stackrel{\mbox{def}}=\lim_{x\uparrow t^2}u(t,x)=1,\qquad (t,x)\in\R^2,\qquad \lim_{x\downarrow -\infty}u(t,x)=0,\qquad t\in\R.
$$
Furthermore, let the function $u_2$ be defined by
$$
u_2(t)=\lim_{x\uparrow t^2}\frac{\partial}{\partial x}u(t,x).
$$
Then the density of (\ref{def_Z}) is given by
\begin{equation}
\label{f_Z}
f_Z(t)=\tfrac12u_2(t)u_2(-t),\,t\in\R.
\end{equation}

The original attempts to compute the density $f_Z$ were based on numerically solving the heat equation above, but it soon became clear that this method did not produce a very accurate solution, mainly because of the rather awkward boundary conditions. However, around 1984 the connection with Airy functions was discovered and this connection was exploited to give analytic solutions in the papers \cite{daniels:85}, \cite{nico:85} and \cite{gro:89}, which were all written in 1984, although the last paper appeared much later.

There seems to be a recent revival of interest in this area of research, see, e.g., \cite{svante:10}, 
\cite{piet:10b}, \cite{piet:11d}, \cite{nico_piet:11}, \cite{leandro:12} and \cite{svante:13}. Also, the main theorem (Theorem 2.3) in \cite{dembo:08}) uses Theorem 3.1 of \cite{gro:89} in an essential way. These recent papers (except \cite{leandro:12}) rely a lot on the results in \cite{daniels:85} and \cite{gro:89}, but it seems fair to say that the derivation of these results in \cite{daniels:85} and \cite{gro:89} is not a simple matter. The most natural approach still seems to use the Cameron-Martin-Girsanov formula for making the transition from Brownian motion with drift to Brownian motion without drift, and next to use the Feynman-Kac formula for determining the distribution of the Radon-Nikodym derivative of the Brownian motion with parabolic drift with respect to the Brownian motion without drift from the corresponding second order differential equation. This is the approach followed in \cite{gro:89}. However, the completion of these arguments used a lot of machinery which one would prefer to avoid. For this reason we give an alternative approach in the present paper.

The starting point of our approach is Theorem 2.1 in \cite{gro:89}, which is given below for convenience. Theorem 2.1 in \cite{gro:89} in fact deals with the process $\{W(t)-ct^2:t\in\R\}$ for an arbitrary positive constant $c>0$, but since we can always deduce the results for general $c$ from the case $c=1$, using Brownian scaling, see, e.g., \cite{svante:13}, we take for convenience $c=1$ in the theorem below. Another simplification is that we consider first hitting times of $0$ for processes starting at $x<0$ instead of first hitting times of $a$ of processes starting at $x<a$ for an arbitrary $a\in\R$, using space homogeneity. We made slight changes of notation, in particular the function $h_x$, $x>0$, of \cite{gro:89} is again denoted by $h_{x}$, but now with a negative argument, so $h_x$ in our paper corresponds to $h_{-x}$ in \cite{gro:89}.

\begin{theorem}[Theorem 2.1 in \cite{gro:89}]
\label{th:stopping_time}
 Let, for $s\in\R$ and $x<0$, $Q^{(s,x)}$ be the probability measure on the Borel $\s$-field of $C([s,\infty):\R)$, corresponding to the process $\{X(t):t\ge s\}$, where $X(t)=W(t)-t^2$, starting at position $x$ at time $s$, and where $\{W(t):t\ge s\}$ is Brownian motion, starting at $x+s^2$ at time $s$. Let the first passage time $\t_0$ of the process $X$ be defined by
$$
\t_0=\inf\{t\ge s:X(t)=0\},
$$
where, as usual, we define $\t_0=\infty$, if $\{t\ge s:X(t)=0\}=\emptyset$. Then
\begin{enumerate}
\item[(i)]
\begin{align*}
Q^{(s,x)}\left\{\t_0\in dt\right\}=e^{-\tfrac23\bigl(t^3-s^3\bigr)+2sx}\psi_{x}(t-s)
E^0\left\{e^{-2\int_0^{t-s}B(u)\,du}\Bigm|B(t-s)=-x\right\}\,dt,
\end{align*}
where $B$ is a \text{\rm Bes(3)} process, starting at zero at time $0$, with corresponding expectation $E^0$, and where $\psi_z(u)=\bigl(2\pi u^3\bigr)^{-1/2}z\exp\bigl(-z^2/(2u)\bigr),\,u,z>0,$ is the value at $u$ of the density of the first passage time through zero of Brownian motion, starting at $z$ at time $0$.
\item[(ii)]
\begin{align*}
Q^{(s,x)}\left\{\t_0\in dt\right\}=e^{-\tfrac23\bigl(t^3-s^3\bigr)+2sx}h_{x}(t-s)\,dt,
\end{align*}
where the function $h_{x}:\R_+\to\R_+$ has Laplace transform
\begin{align*}
\hat h_{x}(\l)=\int_0^{\infty}e^{-\l u}h_{x}(u)\,du
=\ai\bigl(\xi-4^{1/3}x\bigr)/\ai(\xi),\,\quad \xi=2^{-1/3}\l>0,
\end{align*}
and $\ai$ denotes the Airy function $\ai$.
\end{enumerate}
\end{theorem}

\begin{remark}
{\rm Note that the function $h_x$ in the definition of the density of the stopping time $\t_0$ has by part (ii) of Theorem \ref{th:stopping_time} the representation
\begin{equation}
\label{def_h_x}
h_{x}(t)=\frac1{2\pi}\int_{v=-\infty}^{\infty}e^{itv}\frac{\ai(i2^{-1/3}v-4^{1/3}x)}{\ai(i2^{-1/3}v)}\,dv,\qquad\,t>0.
\end{equation}
This representation is obtained by inverting the Laplace transform and will be used in Section \ref{section:max+locmax} and the proof of Lemma 
\ref{lemma:psi_phi}.
}
\end{remark}

\begin{remark}
{ \rm Theorem \ref{th:stopping_time} occurs in different forms in the literature, see, e.g., 
Theorem 2.1 in \cite{paavo:88}. For convenience of the reader, we give a short self-contained proof of Theorem \ref{th:stopping_time} in Appendix A. The interpretation in terms of Bessel process is not really necessary, but this naturally leads to an interpretation in terms of Brownian excursions, further explored in Section 4 of \cite{gro:89}.
}
\end{remark}

Theorem \ref{th:stopping_time} should in principle be sufficient to derive the density $f_Z$ of (\ref{f_Z}), since, defining
$$
q(s)=\lim_{x\uparrow0}\frac{\partial}{\partial x}Q^{(s,x)}\left\{X_t<0,\,\forall t\ge s\right\}
=\lim_{x\uparrow0}\frac{\partial}{\partial x}Q^{(s,x)}\left\{\t_0=\infty\right\},
$$
we find:
$$
f_Z(s)=\tfrac12q(s)q(-s),
$$
following a line of reasoning similar to the derivation of (\ref{f_Z}) in \cite{chernoff:64} (in Chernoff's argument, which is based on a random local perturbation of the starting point $(s,x)$ and the ensuing convolution equation, the factor $1/2$ can be interpreted as the expectation of the squared maximum of the standard Brownian bridge). Moreover, by (ii) of Theorem \ref{th:stopping_time} we have, using inversion of the Laplace transform along the imaginary axis:
\begin{align}
\label{fundamental_relation}
&Q^{(s,x)}\left\{\t_0=\infty\right\}=1-Q^{(s,x)}\left\{\t_0<\infty\right\}
=1-\int_{t=s}^{\infty}Q^{(s,x)}\left\{\t_0\in dt\right\}\nonumber\\
&=1-\int_{t=s}^{\infty}e^{-\tfrac23\bigl(t^3-s^3\bigr)+2sx}
h_{x}(t-s)\,dt\nonumber\\
&=1-\frac{e^{2sx+\tfrac23s^3}}{2\pi}\
\int_{v=-\infty}^{\infty}\frac{\ai\bigl(2^{-1/3}iv-4^{1/3}x\bigr)}{\ai\bigl(2^{-1/3}iv\bigr)}\int_{t=0}^{\infty}e^{itv-\tfrac23(s+t)^3}\,dt\,dv.
\end{align}
So we would be done if we can deal with the properties of the integral in the last line.

However, the latter integral has some unpleasant properties. Taking the special case $s=0$, the integral reduces to:
\begin{align*}
&\frac1{2\pi}\int_{v=-\infty}^{\infty}\frac{\ai\bigl(2^{-1/3}iv-4^{1/3}x\bigr)}{\ai\bigl(2^{-1/3}iv\bigr)}\int_{t=0}^{\infty}e^{ivt-\tfrac23t^3}\,dt\,dv
=\tfrac12\int_{v=-\infty}^{\infty}\frac{\ai\bigl(iv-4^{1/3}x\bigr)}{\ai\bigl(iv\bigr)}\,\hi(iv)\,dv,
\end{align*}
where $\hi$ denotes Scorer's function $\hi$ (the transition of the coefficient $2/3$ of $t^3$ in the left-hand side to the coefficient $1/3$ of $t^3$ in the definition of Scorer's function was made by changes of variables in $t$ and $v$). But to treat the behavior of this integral (and its derivative with respect to $x$) as $x\uparrow0$, we can not take limits inside the integral sign, since we then end up with divergent integrals. For the function $\hi$ has the asymptotic expansion:
$$
\hi(z)\sim -\frac{1}{\pi z}\sum _{{k=0}}^{{\infty}}\frac{(3k)!}{k!(3z^{3})^{k}},\,\qquad\,|\text{ph}(-z)|<\tfrac23\pi-\d
$$
for $\d>0$ arbitrarily small, where $\text{ph}(-z)$ denotes the phase of $-z$, and if we put $x$ equal to zero inside the integral we are stuck with a non-integrable integrand, whereas in fact:
$$
\lim_{x\uparrow0}\tfrac12\int_{v=-\infty}^{\infty}\frac{\ai\bigl(iv-4^{1/3}x\bigr)}{\ai\bigl(iv\bigr)}\,\hi(iv)\,dv=\lim_{x\uparrow0}Q^{(0,x)}\left\{\t_0<\infty\right\}=1.
$$

For this reason part (ii) of Theorem \ref{th:stopping_time} was not directly used in the derivation of density $f_Z$ in \cite{gro:89}, but instead the limit
$$
Q^{(s,x)}\left\{\t_0=\infty\right\}=\lim_{t\to\infty}Q^{(s,x)}\left\{\t_0>t\right\}
$$
was computed by first determining the transition density
$$
Q^{(s,x)}\left\{X_t^{\partial}\in dy\right\},\,t>s,\,x,y<0,
$$
of the process $X_t^{\partial}$, which is the process $X_t$, killed when reaching $0$. The details of this computation were given in the appendix of \cite{gro:89}, giving the result:
\begin{equation}
\label{result89}
Q^{(s,x)}\left\{\t_0=\infty\right\}=\frac{e^{\tfrac23s^3+2sx}}{4^{1/3}}\int_{v=-\infty}^{\infty}e^{-isv}\frac{\ai(i\xi)\bi\bigl(i\xi-4^{1/3}x\bigr)-\ai\bigl(i\xi-4^{1/3}x\bigr)\bi\bigl(i\xi\bigr)}{\ai\bigl(i\xi\bigr)}\,dv,
\end{equation}
where $\xi=2^{-1/3}v$, see Theorem 3.1 of \cite{gro:89}. So by (\ref{fundamental_relation}) we must have the analytic relation
\begin{align}
\label{analytic_relation}
&\frac{e^{\tfrac23s^3+2sx}}{2\pi}\int_{v=-\infty}^{\infty}\frac{\ai\bigl(i\xi-4^{1/3}x\bigr)}{\ai\bigl(i\xi\bigr)}\int_{t=0}^{\infty}e^{itv-\tfrac23(s+t)^3}\,dt\,dv\nonumber\\
&=1-\frac{e^{\tfrac23s^3+2sx}}{4^{1/3}}\int_{v=-\infty}^{\infty}e^{-isv}\frac{\ai\bigl(i\xi\bigr)\bi\bigl(i\xi-4^{1/3}x\bigr)-\ai\bigl(i\xi-4^{1/3}x\bigr)\bi\bigl(i\xi\bigr)}{\ai\bigl(i\xi\bigr)}\,dv,\qquad\xi=2^{-1/3}v.
\end{align}

Conversely, if we can prove the analytic relation (\ref{analytic_relation}), we have an easy road to Theorem 3.1  of \cite{gro:89} and the derivation of the density $f_Z$. We call the function
$$
z\mapsto\frac1{\pi}\int_{t=s}^{\infty}e^{tz-\tfrac13t^3}\,dt
$$
an {\it incomplete Scorer function}, corresponding to the (complete) Scorer function
$$
z\mapsto\hi(z)=\frac1{\pi}\int_{t=0}^{\infty}e^{tz-\tfrac13t^3}\,dt.
$$

In the present paper we first prove in Section \ref{section:analytic_relation} relation (\ref{analytic_relation}) by showing that both integrals, as a function of the parameters $s$ and $x$, satisfy the same extended heat equation. Section \ref{section:max+locmax} discusses the derivation of the distribution of the maximum and location of maximum of one-sided or two-sided Brownian motion with a negative parabolic drift from these results. The appendices contain further details on the results.

\section{A parabolic partial differential equation and the analytic relation (\ref{analytic_relation})}
\label{section:analytic_relation}
\setcounter{equation}{0}

We start with the following lemma.

\begin{lemma}
\label{lemma:function1}
Let the function $f:\R\times(-\infty,0)\to \R$ be defined by
\begin{equation}
\label{an_representation1}
f(s,x)=\frac1{2\pi}\int_{v=-\infty}^{\infty}\frac{\ai\bigl(i\xi-4^{1/3}x\bigr)}{\ai\bigl(i\xi\bigr)}\int_{t=0}^{\infty}e^{itv-\tfrac23(s+t)^3}\,dt\,dv,\qquad \xi=2^{-1/3}v.
\end{equation}
Then $f$ satisfies the partial differential equation
\begin{equation}
\label{pde1}
\frac{\partial}{\partial s}f(s,x)
=-\tfrac12\frac{\partial^2}{\partial x^2}f(s,x)-2xf(s,x).
\end{equation}
Moreover $0\le f(s,x)\le e^{-2sx-\tfrac23s^3}$ and
\begin{equation}
\label{boundary_condition1}
\lim_{x\uparrow0}f(s,x)=e^{-\tfrac23s^3},\qquad \lim_{s\to\infty}f(s,x)=0,\quad\,x<0,\qquad\lim_{x\to-\infty}e^{2sx}f(s,x)=0,\quad\, s\in\R.
\end{equation}
\end{lemma}

\begin{proof} The proof follows from the following observations.\\
\textbf{First Observation:}
Let $\tau_{A}=\inf\{t:X_{t}=A\}$, and 
define $u (s,x)=u(s,x;A)= Q^{(s,x)}\{\tau_{A}<\infty \}$. The process
$\{X_{t}:t\ge s\}$ under $Q^{(s,x)}$ is a diffusion process with a time-dependent
generator, obtained by subtracting $2t(d/dx)$ from the generator of
standard Brownian motion. Consequently, standard arguments from Markov
process theory yield:
\begin{equation}\label{eq:generatorEqn}
	\frac{\partial u}{\partial s} = -\frac{1}{2}
	\frac{\partial^{2}u}{\partial x^{2}} +2s \frac{\partial
	u}{\partial x}
\end{equation}
in the region $x<A$. 

\noindent \textbf{Second Observation:}
Let $u (s,x)$ and $f (s,x)$ be functions satisfying the relation 
\begin{equation}\label{eq:expTilt}
	f (s,x)=e^{-2sx-\tfrac23s^3}u (s,x).
\end{equation}
Then $u$ satisfies the PDE
\eqref{eq:generatorEqn} if and only if $f$ satisfies the PDE (\ref{pde1}),
as can be seen by routine calculus.\\
\noindent \textbf{Third Observation:} Relation (\ref{fundamental_relation})
(which follows from Theorem \ref{th:stopping_time} by the Laplace inversion (\ref{def_h_x})) shows that the function $u
(s,x)=Q^{(s,x)}\{\tau_{0}<\infty \}$ is related to the function $f
(s,x)$ of the lemma by the transformation \eqref{eq:expTilt}
above. Since $u$ satisfies \eqref{eq:generatorEqn}, it now follows
immediately that $f$ satisfies (\ref{pde1}).\\
The boundary conditions follow immediately from the probabilistic interpretation of the function $u$.
\end{proof}

\vspace{0.3cm}
It turns out that the right-hand side of (\ref{analytic_relation}) has a more convenient representation, which generalizes relation (2.3) of Lemma 2.2 in \cite{nico_piet:11} (see also Remark 2.1 in \cite{nico_piet:11} on the equivalent relation (5.10) in \cite{svante:10}).

\begin{lemma}
\label{lemma:function2_rep}
Let the function $g:\R\times(-\infty,0]\to \R$ be defined by
\begin{equation}
\label{an_representation2}
g(s,x)=\frac1{4^{1/3}}\int_{v=-\infty}^{\infty}e^{-isv}\frac{\ai(i\xi)\bi\bigl(i\xi-4^{1/3}x\bigr)-\ai\bigl(i\xi-4^{1/3}x\bigr)\bi(i\xi)}{\ai(i\xi)}\,dv,\qquad \xi=2^{-1/3}v.
\end{equation}
Then $g$ has the alternative representation
\begin{equation}
\label{function2_rep}
g(s,x)=\frac{e^{-2sx}}{2\pi}\int_{u=-\infty}^{\infty}\frac{\int_{y=0}^{-4^{1/3}x}e^{-2^{1/3}s(iu+y)}\ai(iu+y)\,dy}{\ai(iu)^2}\,du.
\end{equation}
\end{lemma}

\begin{proof}
By the definition of the function $g$, we have:
\begin{align*}
g(s,x)=2^{-1/3}\int_{-\infty}^{\infty}e^{-2^{1/3}isu}\frac{\ai(iu)\bi\bigl(iu-4^{1/3}x\bigr)-\bi(iu)\ai\bigl(iu-4^{1/3}x\bigr)}{\ai(iu)}\,du.
\end{align*}
For simplicity of notation, we consider instead:
\begin{align*}
\tilde g(s,x)&=\tfrac12\int_{-\infty}^{\infty}e^{-isu}\frac{\ai(iu)\bi(iu+x)-\bi(iu)\ai(iu+x)}{\ai(iu)}\,du
,\quad s\in\R,\,x>0.
\end{align*}
It is shown in Section \ref{section:appendixB} that the function $x\mapsto \tilde g(s,x)$ satisfies the first order differential equation:
\begin{align}
\label{DE_g}
\frac{\partial}{\partial x}\tilde g(s,x)=s\tilde g(s,x)+\frac1{2\pi}\int_{u=-\infty}^{\infty}e^{-isu}\frac{\ai(iu+x)}{\ai(iu)^2}\,du.
\end{align}
So if $\tilde g(s,0)=0$, the solution is given by:
$$
\tilde g(s,x)=\frac{e^{sx}}{2\pi}\int_{u=-\infty}^{\infty}\frac{\int_{y=0}^x e^{-s(iu+y)}\ai(iu+y)\,dy}{\ai(iu)^2}\,du.
$$
Transferring this result to the function $g$ and using $g(s,0)=0$, we get that the corresponding linear differential equation for $x\mapsto g(s,x)$ has the solution given by (\ref{function2_rep}).
\end{proof}

\begin{lemma}
\label{lemma:function2}
Let the function $g:\R\times(-\infty,0]\to \R$ be defined as in Lemma \ref{lemma:function2_rep}.
Then $g$ satisfies the partial differential equation
\begin{equation}
\label{pde2}
\frac{\partial}{\partial s}g(s,x)
=-\tfrac12\frac{\partial^2}{\partial x^2}g(s,x)-2xg(s,x).
\end{equation}
Moreover:
\begin{equation}
\label{boundary_condition2}
\lim_{x\uparrow0}g(s,x)=0,\qquad\lim_{x\to-\infty}e^{2sx}g(s,x)=e^{-\tfrac23s^3},\qquad s>0.
\end{equation}
\end{lemma}

\begin{proof}
We have:
\begin{align*}
&\frac{\partial^2}{\partial x^2}\frac{\ai(i\xi)\bi\bigl(i\xi-4^{1/3}x\bigr)-\ai\bigl(i\xi-4^{1/3}x\bigr)\bi(i\xi)}{\ai(i\xi)}\\
&=2(iv-2x)\frac{\ai(i\xi)\bi\bigl(i\xi-4^{1/3}x\bigr)-\ai\bigl(i\xi-4^{1/3}x\bigr)\bi(i\xi)}{\ai(i\xi)}\,,\qquad \xi=2^{-1/3}v.
\end{align*}
We also have:
$$
\frac{\partial}{\partial s}g(s,x)=-\frac1{4^{1/3}}\int_{v=-\infty}^{\infty}ive^{-isv}\frac{\ai(i\xi)\bi\bigl(i\xi-4^{1/3}x\bigr)-\ai\bigl(i\xi-4^{1/3}x\bigr)\bi(i\xi)}{\ai(i\xi)}\,dv\,.
$$
This yields (\ref{pde2}).
It is clear from the definition (\ref{an_representation2}) that $\lim_{x\uparrow0}g(s,x)=0$ for all $s\in\R$. A stronger version of the second part of (\ref{boundary_condition2}) is proved in Section \ref{section:appendixC}.
\end{proof}

The preceding two lemmas  give the desired result (\ref{analytic_relation}).

\begin{theorem}
\label{th:solution_relation}
\begin{enumerate}
\item[(i)]
Let the functions $f$ and $g$ be defined as in Lemmas \ref{lemma:function1} to \ref{lemma:function2}. Then we have:
$$
f(s,x)=e^{-2sx-\tfrac23s^3}-g(s,x),\quad s\in\R,\quad x\le0,
$$
where $f(s,0)$ is defined by taking the limit of $f(s,x)$, as $x\uparrow0$.
\item[(ii)] Let, for $s\in\R$ and $x\le0$, $\{X_t:t\in\R\}=\{W(t)-t^2:t\in\R\}$ be Brownian motion with a negative parabolic drift, starting at $x$ at time $s$, with corresponding probability measure $Q^{(s,x)}$. Then
\begin{equation}
\label{rep1}
Q^{(s,x)}\{\t_0<\infty\}=e^{2sx+\tfrac23s^3}f(s,x),
\end{equation}
and
\begin{equation}
\label{rep2}
Q^{(s,x)}\{\t_0=\infty\}=e^{2sx+\tfrac23s^3}g(s,x)=\frac{e^{\tfrac23s^3}}{2\pi}\int_{u=-\infty}^{\infty}\frac{\int_{y=0}^{-4^{1/3}x}e^{-2^{1/3}s(iu+y)}\ai(iu+y)\,dy}{\ai(iu)^2}\,du.
\end{equation}
\end{enumerate}

\end{theorem}

\begin{proof} (i). The function
$$
(s,x)\mapsto e^{-2sx-\tfrac23s^3},\quad\,(s,x)\in\R^2,
$$
satisfies the same partial differential equation as the functions $f$ and $g $ of Lemmas \ref{lemma:function1} and \ref{lemma:function2}. We have to show:
\begin{equation}
\label{maximum_princ}
h(s,x)\stackrel{\text{\small def}}=f(s,x)+g(s,x)-e^{-2sx-\tfrac23s^3}=0,\qquad s\in\R,\qquad x\le0,
\end{equation}
defining $f(s,0)$ and $g(s,0)$ by the limits of $f(s,x)$ and $g(s,x)$ as $x\uparrow0$, respectively. To show that (\ref{maximum_princ}) holds, we use the maximum principle.

First of all, (\ref{maximum_princ}) holds for all $s$ if $x=0$ by Lemmas \ref{lemma:function1} and \ref{lemma:function2}.
It is shown in Section \ref{section:appendixC} that also
\begin{equation}
\label{max_bound1}
\lim_{x\to-\infty}h(s,x)=0,\qquad\forall s\in\R,
\end{equation}
and
\begin{equation}
\label{max_bound2}
\lim_{s\to\infty}h(s,x)=0,\qquad\forall x<0.
\end{equation}
We now consider an infinite rectangle $R_c=\{(s,x) : s\ge c, x \le 0\}$, for some $c\in\R$. Suppose that $h$ attains a strictly positive maximum over $R_c$ at an interior point $(s_0,x_0)\in R_c^0$. Then $\partial_1h(s_0,x_0)=0$, denoting the derivative w.r.t.\ the $i$th argument by $\partial_i$. Hence, since $h$ satisfies the same partial differential equation as $f$ and $g$, we get:
$$
0=\partial_1h(s_0,x_0)=-\tfrac12\partial_2^2h(s_0,x_0)-2x_0h(s_0,x_0),
$$
implying
$$
\partial_2^2h(s_0,x_0)=-4x_0h(s_0,x_0)>0,
$$
since $x_0<0$. But this contradicts the assumption that $h$ attains its maximum at $(s_0,x_0)$. Similarly, if $h$ attains a strictly negative minimum at an interior point $(s_0,x_0)\in R_c^0$, we would get $\partial_2^2h(s_0,x_0)<0$, again giving a contradiction. So a strictly positive maximum or strictly negative minimum over $R_c$ can only be attained on the line $s=c$. Suppose that a strictly positive maximum is attained at the point $(c,x_0)$, where $x_0<0$. Then we must have:
$\partial_1 h(c,x_0)\le0$, implying by the partial differential equation for $h$:
$$
\partial_2^2 h(c,x_0)\ge -4x_0h(c,x_0)>0,
$$
contradicting the assumption that $h$ attains its maximum on the line $s=c$ at the point $(c,x_0)$.

In a similar way we get a contradiction if we assume that $h$ attains a strictly negative minimum on the line $s=c$. So the conclusion is that $h$ is identically zero on $R_c$. Since the argument holds for all $c\in\R$, we get that the function $h$ is identically zero on $\R\times(-\infty,0]$.\\
(ii) This follows from (\ref{fundamental_relation}), Lemmas \ref{lemma:function1} to \ref{lemma:function2}, and (i).
\end{proof}

\section{The distribution of the maximum and location of maximum of one-sided and two-sided Brownian motion with parabolic drift.}
\label{section:max+locmax}
\setcounter{equation}{0}
Let $M$ denote the maximum of the process $\{X_t=W(t)-t^2:t\ge s\}$, starting at $x$ at time $s$, with corresponding probability measure $Q^{(s,x)}$. Moreover, let, with a slight abuse of notation, $\t_M$ denote the location of the maximum $M$ of this process. The following theorem gives the joint distribution of $\t_M$ and $M$ under $Q^{(s,x)}$.

\begin{theorem}
\label{th:max+locmax}
Let the function $k$ be defined by
\begin{equation}
\label{def_k}
k(s,x)=\frac{\partial}{\partial x}Q^{(s,x)}\left\{\t_0<\infty\right\},
\end{equation}
where $Q^{(s,x)}$ is the probability measure, corresponding to the process $\{X_t=W(t)-t^2:t\ge s\}$, starting at $x$ at time $s$. Moreover, let  $k(t,0)=\lim_{x\uparrow0}k(t,x)$ for all $t\in\R$.
Then
\begin{enumerate}
\item[(i)]
\begin{equation}
\label{stopping_time_dist}
Q^{(s,x)}\left\{\t_0<\infty\right\}
=e^{\tfrac23s^3+2sx}f(s,x)=1-e^{\tfrac23s^3+2sx}g(s,x)
\end{equation}
where the functions $f$ and $g$ are defined as in Lemma \ref{lemma:function1} and Lemma \ref{lemma:function2}, respectively, and 
\begin{equation}
\label{k(s,0)}
k(s,0)=\lim_{x\downarrow0}\frac{\partial}{\partial x}Q^{(s,x)}\left\{\t_0<\infty\right\}=\frac{e^{\tfrac23s^3}}{4^{1/3}\pi}\int_{v=-\infty}^{\infty}\frac{e^{-isv}}{\ai(i2^{-1/3}v)}\,dv.
\end{equation}
\item[(ii)]
The function $a\mapsto k(s,x-a)$ is the density of the maximum $M$ at $a>x$ under the probability measure $Q^{(s,x)}$.
\item[(iii)]
The joint density of $\t_M$ and $M$ is given by:
\begin{align}
\label{joint1}
f_{(\t_M,M)}(t,a)
&=e^{-\tfrac23\bigl(t^3-s^3\bigr)+2s(x-a)}h_{x-a}(t-s)k(t,0)\nonumber\\
&=\frac{e^{-\tfrac23s^3+2s(x-a)}h_{x-a}(t-s)}{\pi}\int_{v=-\infty}^{\infty}\frac{e^{-itv}}{\ai(i\xi)}\,dv,\quad a>x,\quad t>s,
\end{align}
where $h_{x-a}$ is defined as in part (ii) of Theorem \ref{th:stopping_time}, that is:
$$
h_{x-a}(u)=\frac1{2\pi}\int_{v=-\infty}^{\infty}e^{iuv}
\frac{\ai\bigl(i2^{-1/3}v-4^{1/3}(x-a)\bigr)}{\ai\bigl(i2^{-1/3}v\bigr)}\,dv.
$$
\end{enumerate}
\end{theorem}

\begin{proof}
(i) (\ref{stopping_time_dist}) is relation (\ref{analytic_relation}), which follows from Theorem \ref{th:solution_relation} in Section \ref{section:analytic_relation}, and (\ref{k(s,0)}) follows from the representation in the right-hand side of (\ref{analytic_relation}) by taking the derivative w.r.t.\ $x$, letting $x\downarrow0$ and using that the Wronskian of the two solutions $\ai$ and $\bi$ of the Airy differential equation equals $1/\pi$.\\
(ii) By a space homogeneity argument, the density of the maximum $M$ under $Q^{(s,x)}$ is given by
\begin{align*}
&-\frac{\partial}{\partial a}Q^{(s,x)}\{M>a\}=-\frac{\partial}{\partial a}Q^{(s,x)}\{\t_a<\infty\}
=-\frac{\partial}{\partial a}Q^{(s,x-a)}\{\t_0<\infty\}\\
&=\frac{\partial}{\partial x}Q^{(s,x-a)}\{\t_0<\infty\}=k(s,x-a).
\end{align*}
(iii)
Since, if the process starts at $(s,x)$, with $s<t$ and $x<a$, we only can have $\t_M<t$ and $M<a$ if $M=b\in(x,a)$ and $\t_M=u\in(s,t)$, we have:
\begin{align*}
Q^{(s,x)}\left\{\t_M<t,\,M<a\right\}=\int_{b=x}^a\int_{u=s}^t Q^{(s,x)}\left\{\t_b\in du\right\}k(u,0)\,db,
\end{align*}
where $k(u,0)$ corresponds to the event that the maximum of the path, started at $b$ at time $u$, stays below $b$. Hence differentiation gives:
\begin{align*}
f^{(s,x)}_{\t_M,M}(t,a)\,dt&= Q^{(s,x)}\left\{\t_a\in dt\right\}k(t,0)
=Q^{(s,x-a)}\left\{\t_0\in dt\right\}k(t,0)\\
&=e^{-\tfrac23\bigl(t^3-s^3\bigr)+2s(x-a)}h_{x-a}(t-s)k(t,0)\,dt,
\end{align*}
where we use part (ii) of Theorem \ref{th:stopping_time} in the last equality.
\end{proof}

As a corollary we get the corresponding result for two-sided Brownian motion.

\begin{corollary}
\label{cor:loc+locmax_BM}
Let $X_t=W(t)-t^2$, where $\{W(t):t\in\R\}$ is two-sided Brownian motion, originating from zero. Furthermore, let $M$ and $\t_M$ be the maximum and the location of the maximum of the process $\{X_t:t\in\R\}$, respectively. Then the joint density of $(\t_M,M)$ is given by
\begin{equation}
\label{joint-2sided}
f_{(\t_M,M)}(t,a)=h_{-a}(|t|)g(0,-a)\f(|t|),
\end{equation}
where $h_{-a}$ is defined as in part (ii) of Theorem \ref{th:stopping_time}, $g(t,-a)$ by (\ref{an_representation2})
of Lemma \ref{lemma:function2}, and $\phi$ by:
\begin{equation}
\label{phi}
\f(t)=\frac{1}{4^{1/3}\pi}\int_{v=-\infty}^{\infty}\frac{e^{-itv}}{\ai(i2^{-1/3}v)}\,dv,\,t\in\R.
\end{equation}
\end{corollary}

\begin{proof}
Let $t>0$ and let  $\M_+$ and $\t_{M+}$ be the maximum and the location of the maximum for the one-side process to the right of zero. By part (iii) of Theorem \ref{th:max+locmax}, the density of $(\t_{M+},M_+)$ is given by (\ref{joint1}), which, since $s=x=0$, boils down to
$$
\frac{h_{-a}(t)}{4^{1/3}\pi}\int_{v=-\infty}^{\infty}\frac{e^{-itv}}{\ai(i\xi)}\,dv.
$$
If we want to turn this into the density of the global maximum and location of maximum on $\R$, we have to multiply this density with the probability that the maximum left of zero is less than $M+$, which means, using a symmetry argument, that the density becomes:
$$
f_{(\t_M,M)}(t,a)=\frac{h_{-a}(t)Q^{(0,0)}\{\t_a=\infty\}}{4^{1/3}\pi}\int_{v=-\infty}^{\infty}\frac{e^{-itv}}{\ai(i\xi)}\,dv
=\frac{h_{-a}(t)Q^{(0,-a)}\{\t_0=\infty\}}{4^{1/3}\pi}\int_{v=-\infty}^{\infty}\frac{e^{-itv}}{\ai(i\xi)}\,dv,
$$
where $\xi=2^{-1/3}v$. By (\ref{stopping_time_dist}) of Theorem \ref{th:max+locmax} we now get:
$$
f_{(\t_M,M)}(t,a)
=\frac{h_{-a}(t)g(0,-a)}{4^{1/3}\pi}\int_{v=-\infty}^{\infty}\frac{e^{-itv}}{\ai(i\xi)}\,dv
=h_{-a}(t)g(0,-a)\f(t).
$$
The case where the maximum is reached to the left of zero is treated in a similar way.
\end{proof}

\begin{remark}
{\rm
Note that the function $\f$, defined by (\ref{phi}), has the following probabilistic interpretation:
\begin{equation}
\label{phi_interpretation}
\f(t)=-e^{-\tfrac23t^3}\frac{\partial}{\partial x}Q^{(t,x)}\{\t_0=\infty\}\biggr|_{x=0}.
\end{equation}
This interpretation can perhaps easiest be seen from the representation (\ref{rep2}) 
in Theorem \ref{th:max+locmax}. The function defines the density of the location of the maximum, as is seen in the following Corollary \ref{cor:Chernoff}.
}
\end{remark}

\begin{corollary}
\label{cor:Chernoff}
Let $X_t=W(t)-t^2$, where $\{W(t):t\in\R\}$ is two-sided Brownian motion, originating from zero. Then the density of the location of the maximum $\t_M$ is given by:
\begin{equation}
f_{\t_M}(t)=\tfrac12\f(t)\f(-t),
\end{equation}
where $\f$ is defined by (\ref{phi}).
\end{corollary}

\begin{proof}
We have by Corollary \ref{cor:loc+locmax_BM} and Theorem \ref{th:max+locmax} for $t\in\R$:
\begin{align*}
f_{\t_M}(t)&=\f(t)\int_{x=0}^{\infty}h_{-x}(t)g(0,-x)\,dx.
\end{align*}
Now let $\psi$ be defined by
\begin{equation}
\label{def_psi}
\psi(t)=\int_{x=0}^{\infty}h_{-x}(t)g(0,-x)\,dx.
\end{equation}
Then we have to show:
\begin{equation}
\label{key_Chernoff}
\psi(t)=\tfrac12\f(-t).
\end{equation}
Since, using a time reversal argument, the density obviously has to be symmetric, we only have to prove (\ref{key_Chernoff}) for all $t\ge0$. 
The equality is derived in the proof of Lemma \ref{lemma:psi_phi} in Section
\ref{section:appendixD} by an (asymptotic) analytic argument.
\end{proof}

\begin{remark}
{\rm
Note that the integrand on the right-hand side of (\ref{def_psi}) is the product of the density of the first hitting time $\t_{x}$ under $Q^{(0,0)}$ and the probability that the drifting process stays below $x$ under $Q^{(0,0)}$. The latter probability can also be interpreted as the probability that the process $\{X_t=W(t)-t^2,\,t\le0\}$, starting at zero and running to the left, stays below $x$. So, intuitively, the product $h_{-x}(t)g(0,-x)\f(t)$ corresponds to paths of two-sided Brownian motion minus a parabola, having their first hitting time of $x$ at time $t>0$ (the factor $h_{-x}(t)g(0,-x)$), and staying below $x$ on the interval $[t,\infty)$ (the factor $\f(t)$). The factor $\exp\{-\tfrac23t^3\}$ in (\ref{phi_interpretation}) disappears, since by part (ii) of Theorem \ref{th:stopping_time},
$$
h_{-x}(t)\,dt=e^{\tfrac23t^3}Q^{(0,x)}\left\{\t_0\in dt\right\}.
$$
The factor $\tfrac12$ in front of the product is introduced by going from a derivative in the space variable $x$ in (\ref{phi_interpretation}) to a derivative in the time variable $t$. This seems a bit different from the way the factor $\tfrac12$ entered in Chernoff's argument as the expectation of the squared maximum of the Brownian bridge.}
\end{remark}

\section{Concluding remarks}
\label{section:conclusion}
\setcounter{equation}{0}

We gave a direct approach to Chernoff's theorem and other results of this type, using the Feynman-Kac formula with a stopping time and the analytic relation (\ref{analytic_relation}). Relation (\ref{analytic_relation}) is proved by showing that the integrals in this relation satisfy the parabolic partial differential equation (\ref{pde1}) as a function of the  parameters $s$ (time) and $x$ (space) and by an application of the maximum principle. As shown in \cite{piet:10b} and  \cite{svante:10}, these results also give the distribution of the maximum of Brownian motion minus a parabola itself, both for the one-sided and two-sided case. An asymptotic development of the tail of the distribution of this maximum is given in \cite{nico_piet:11}. We hope that the direct approach of the present paper will make these results more accessible. The original proofs in \cite{gro:89} were rather long and technical, and lacked this property.

It is proved in \cite{piet:11d} and again in \cite{svante:13} that the maximum $M$ and the location of the maximum $\t_M$ of two-sided Brownian motion minus the parabola $t^2$ satisfy the relation
$$
E\t_M^2=\tfrac13EM.
$$
This result is generalized in \cite{leandro:12}, where the relation is proved not using Airy functions, and where a completely general result of this type is given for drifting Brownian motion. More results for moments and combinatorics for Airy integrals are given in \cite{svante:13}. The latter paper ends with a series of conjectures and open problems in this area.

For computational purposes, the representation (\ref{function2_rep}) in Lemma \ref{lemma:function2_rep} seems the best choice, since we lose in this way the inconvenient difference of products of the Airy functions $\ai$ and $\bi$ on the right side of (\ref{analytic_relation}), which are not integrable along the imaginary axis by themselves, and we also do not have the trouble near zero that the function on the left of relation (\ref{analytic_relation}), further analyzed in Lemma \ref{lemma:function1}, is exhibiting.

\section{Appendix A}
\label{section:FK}
\setcounter{equation}{0}
In this section we prove Theorem \ref{th:stopping_time}. We start with the Feynman-Kac part.

Let, for $\l>0$, $u_{\l}$ be the unique non-negative solution of the boundary problem
\begin{equation}
\label{Airy-equation}
\tfrac12u''(x)-\bigl(\l-2x\bigr)u(x)=0,\,x<0,\qquad\lim_{x\uparrow0} u(x)=1,\qquad u(x)\le1,\, x\le0.
\end{equation}
The unique solution of (\ref{Airy-equation}) is given by
\begin{equation}
\label{u_lambda}
u_{\l}(x)=\frac{\ai\bigl(2^{-1/3}\l-4^{1/3}x\bigr)}{\ai\bigl(2^{-1/3}\l\bigr)},\,x\le0.
\end{equation}
We now consider the process
$$
Y_t=e^{-\int_{v=s}^t \bigl(\l-2X_v\bigr)\,dv}u_{\l}(X_t),
$$
where $X_t$ is standard Brownian motion, starting at $x<0$ at time $0$. By It\^o's formula and (\ref{Airy-equation}) we have:
$$
du(X_t)=u_{\l}'(X_t)dX_t+\tfrac12u_{\l}''(X_t)\,dt=u_{\l}'(X_t)dX_t+\bigl(\l-2X_t\bigr)u_{\l}(X_t)\,dt.
$$
So we get:
\begin{align*}
dY_t&=-u_{\l}(X_t)e^{-\int_{v=0}^t \bigl(\l-2X_v\bigr)\,dv}\bigl(\l-2X_t\bigr)\,dt\\
&\qquad\qquad+e^{-\int_{v=0}^t \bigl(\l-2X_v\bigr)\,dv}\left\{u_{\l}'(X_t)dX_t+\bigl(\l-2X_t\bigr)u_{\l}(X_t)\,dt\right\}\\
&=e^{-\int_{v=0}^t \bigl(\l-2X_v\bigr)\,dv}u_{\l}'(X_t)\,dX_t,
\end{align*}
implying that $Y_t$ is a local martingale and that
$$
Y_{\t_0}-x=\int_{t=0}^{\t_0}e^{-\int_{v=0}^t \bigl(\l-2X_v\bigr)\,dv}u_{\l}'(X_t)\,dX_t.
$$
Moreover, $t\mapsto Y_{t\wedge \tau_0}$ is a bounded martingale, and hence:
$$
E^{x}e^{-\int_{v=0}^{\t_0}\bigl(\l-2X_v\bigr)\,dv}
=E^{x}e^{-\int_{v=0}^{\t_0}\bigl(\l-2X_v\bigr)\,dv}u_{\l}(X_{\t_0})=E^{x}Y_{\t_0}=Y_s=u_{\l}(x),
$$
with the conclusion that
$$
\int_{t=0}^{\infty}e^{-\l t}E^x\left\{e^{\int_{v=0}^{t} 2X_v\,dv}\Bigm|\t_0=t\right\}P^{x}\left\{\t_0\in dt\right\}=u_{\l}(x)=\frac{\ai\bigl(2^{-1/3}\l-4^{1/3}x\bigr)}{\ai\bigl(2^{-1/3}\l\bigr)},\quad x<0,
$$
where $P^x$ denotes the probability measure of Brownian motion in standard scale, starting at $x$ at time $0$.

We now turn to the Cameron-Martin-Girsanov part of the proof. Suppose that $P^{(s,x)}$ and $Q^{(s,x)}$ are the probability measures of continuous paths from $(s,x)$ such that:
\begin{enumerate}
\item[(1)] under $P^{(s,x)}$, $\{X_t:t\ge s\}$ is standard Brownian motion with $X_s=x$,
\item[(2)] under $Q^{(s,x)}$, $\{X_t:t\ge s\}$ is standard Brownian motion $-t^2$, with $X_s=x$.
\end{enumerate}
Then, by the Cameron-Martin-Girsanov formula, $Q^{(s,x)}<<P^{(s,x)}$ on $\{{\cal F}_t:t\ge s\}$, where ${\cal F}_t=\s\{X_u:u\in[s,t]\}$ and
$$
\frac{dQ^{(s,x)}}{dP^{(s,x)}}\Bigr|_{{\cal F}_t}=Z_t,
$$
where
$$
Z_t=\exp\left\{-2\int_s^t u\,dX_u-\tfrac23\bigl(t^3-s^3\bigr)\right\}=
\exp\left\{2\int_s^t X_u\,du-2\bigl(tX_t-sX_s\bigr)-\tfrac23\bigl(t^3-s^3\bigr)\right\},\quad t\ge s.
$$
This implies that we have:
\begin{align*}
&Q^{(s,x)}\left\{\t_0\in dt\right\}\\
&=\exp\left\{2sx-\tfrac23\bigl(t^3-s^3\bigr)\right\}E^{P^{(s,x)}}\left\{\exp\left\{2\int_s^t X_u\,du\right\}\bigm|\t_0=t\right\}
P^{(s,x)}\left\{\t_0\in dt\right\}.
\end{align*}
This is part (i) of Theorem \ref{th:stopping_time}. The interpretation in terms of a Bessel process, given in (i) of Theorem \ref{th:stopping_time} is standard (and further detailed in \cite{gro:89}).

Moreover, by time homogeneity and the Feynman-Kac argument, given above, we have
\begin{align*}
&E^{P^{(s,x)}}\left\{\exp\left\{2\int_s^t X_u\,du\right\}\bigm|\t_0=t\right\}
P^{(s,x)}\left\{\t_0\in dt\right\}\\
&=E^{P^{(0,x)}}\left\{\exp\left\{2\int_0^{t-s} X_u\,du\right\}\bigm|\t_0=t-s\right\}
P^{(0,x))}\left\{s+\t_0\in dt\right\}=h_x(t-s)\,dt
\end{align*}
where the function $h_{x}:\R_+\to\R_+$ has Laplace transform
\begin{align*}
\hat h_{x}(\l)=\int_0^{\infty}e^{-\l u}h_{x}(u)\,du
=\ai\bigl(\xi-4^{1/3}x\bigr)/\ai(\xi),\,\quad \xi=2^{-1/3}\l>0.
\end{align*}
This gives part (ii) of Theorem \ref{th:stopping_time}.

\section{Appendix B}
\label{section:appendixB}
\setcounter{equation}{0}

\begin{lemma}
\label{lemma:DE_g}
Let, for $s\in\R$ and $x\ge0$, the function $\tilde g$ be defined by
\begin{equation}
\label{def_tilde_g}
\tilde g(s,x)=\tfrac12\int_{-\infty}^{\infty}e^{-isu}\frac{\ai(iu)\bi(iu+x)-\bi(iu)\ai(iu+x)}{\ai(iu)}\,du.
\end{equation}
Then, for each $s\in\R$, the function $x\mapsto\tilde g(s,x)$ satisfies the differential equation
\begin{equation}
\label{DE_g2}
\frac{\partial}{\partial x}\tilde g(s,x)=s\tilde g(s,x)+\frac1{2\pi}\int_{u=-\infty}^{\infty}e^{-isu}\frac{\ai(iu+x)}{\ai(iu)^2}\,du.
\end{equation}
\end{lemma}

\begin{proof}
We have:
\begin{align*}
\tilde g(s,x)&=\tfrac12\int_{-\infty}^{\infty}e^{-isu}\frac{\ai(iu)\bi(iu+x)-\bi(iu)\ai(iu+x)}{\ai(iu)}\,du\\
&=e^{-i\pi/6}\int_{0}^{\infty}e^{-isu}\left\{\ai\bigl(e^{-i\pi/6}u+e^{-2i\pi/3}x\bigr)-\frac{\ai\bigl(e^{-i\pi/6}u\bigr)\ai(iu+x)}{\ai(iu)}\right\}\,du\\
&\qquad\qquad+e^{i\pi/6}\int_{0}^{\infty}e^{isu}\left\{\ai\bigl(e^{i\pi/6}u+e^{2i\pi/3}x\bigr)-\frac{\ai\bigl(e^{i\pi/6}u\bigr)\ai(-iu+x)}{\ai(-iu)}\right\}\,du
,\quad s\in\R,\quad x>0.
\end{align*}
We now get:
\begin{align*}
&\frac{\partial}{\partial x}e^{-i\pi/6}\int_{0}^{\infty}e^{-isu}\ai\bigl(e^{-i\pi/6}u+e^{-2i\pi/3}x\bigr)\,du
=e^{-5i\pi/6}\int_{0}^{\infty}e^{-isu}\ai'\bigl(e^{-i\pi/6}u+e^{-2i\pi/3}x\bigr)\,du\\
&=\left[e^{-2i\pi/3}e^{-isu}\ai\bigl(e^{-i\pi/6}u+e^{-2i\pi/3}x\bigr)\right]_{u=0}^{\infty}
+ise^{-2i\pi/3}\int_{0}^{\infty}e^{-isu}\ai\bigl(e^{-i\pi/6}u+e^{-2i\pi/3}x\bigr)\,du\\
&=-e^{-2i\pi/3}\ai\bigl(e^{-2i\pi/3}x\bigr)
+ise^{-2i\pi/3}\int_{0}^{\infty}e^{-isu}\ai\bigl(e^{-i\pi/6}u+e^{-2i\pi/3}x\bigr)\,du.
\end{align*}
Similarly,
\begin{align*}
&\frac{\partial}{\partial x}e^{i\pi/6}\int_{-\infty}^{0}e^{-isu}\ai\bigl(e^{i\pi/6}u+e^{2i\pi/3}x\bigr)\,du\\
&=-e^{2i\pi/3}\ai\bigl(e^{2i\pi/3}x\bigr)
-ise^{2i\pi/3}\int_{0}^{\infty}e^{isu}\ai\bigl(e^{i\pi/6}u+e^{2i\pi/3}x\bigr)\,du.
\end{align*}
This implies:
\begin{align*}
&\frac{\partial}{\partial x}
\left\{e^{-i\pi/6}\int_{0}^{\infty}e^{-isu}\ai\bigl(e^{-i\pi/6}u+e^{-2i\pi/3}x\bigr)\,du
+e^{i\pi/6}\int_{0}^{\infty}e^{isu}\ai\bigl(e^{i\pi/6}u+e^{2i\pi/3}x\bigr)\,du\right\}\\
&=\ai(x)+se^{-i\pi/6}\int_{0}^{\infty}e^{-isu}\ai\bigl(e^{-i\pi/6}u+e^{-2i\pi/3}x\bigr)\,du
+se^{i\pi/6}\int_{0}^{\infty}e^{isu}\ai\bigl(e^{i\pi/6}u+e^{2i\pi/3}x\bigr)\,du,
\end{align*}
using
$$
\ai(x)=-e^{-2i\pi/3}\ai\bigl(e^{-2i\pi/3}x\bigr)-e^{2i\pi/3}\ai\bigl(e^{2i\pi/3}x\bigr).
$$
Furthermore,
\begin{align*}
&e^{-i\pi/6}\int_{0}^{\infty}e^{-isu}\frac{\ai\bigl(e^{-i\pi/6}u\bigr)\ai'(iu+x)}{\ai(iu)}\,du\\
&=\left[e^{-i\pi/6}e^{-isu}\frac{\ai\bigl(e^{-i\pi/6}u\bigr)\ai(iu+x)}{i\ai(iu)}\right]_{u=0}^{\infty}
-e^{-i\pi/6}\int_{0}^{\infty}\ai(iu+x)\frac{d}{du}\left(e^{-isu}\frac{\ai\bigl(e^{-i\pi/6}u\bigr)}{i\ai(iu)}\right)\,du\\
&=e^{i\pi/3}\ai(x)-\int_{0}^{\infty}\ai(iu+x)\frac{d}{du}\left(\frac{e^{-i\pi/6}\ai\bigl(e^{-i\pi/6}u\bigr)}{i\ai(iu)}\right)\,du\\
&\qquad\qquad\qquad\qquad\qquad\qquad+se^{-i\pi/6}\int_{0}^{\infty}e^{-isu}\frac{\ai\bigl(e^{-i\pi/6}u\bigr)\ai(iu+x)}{i\ai(iu)}\,du.
\end{align*}
Since
\begin{align}
\label{h(u)}
h(u)\stackrel{\text{\small def}}=\frac{d}{du}\left(\frac{e^{-i\pi/6}\ai\bigl(e^{-i\pi/6}u\bigr)}{i\ai(iu)}\right)=\frac1{2\pi\ai(iu)^2}\,,
\end{align}
(see the proof of Lemma 2.2 in \cite{nico_piet:11}), the conclusion now follows.
\end{proof}

\section{Appendix C}
\label{section:appendixC}
\setcounter{equation}{0}

We first note that, for $s=0$, the limit of $g(s,x)$, as $x\to-\infty$, has been studied in \cite{nico_piet:11}, since it describes the behavior of the maximum of one-sided Brownian motion minus a parabola, starting at time $s=0$. However, this analysis started from the probabilistic interpretation of the right-hand side of relation (\ref{analytic_relation}), which is something we cannot do at this point, since we want to derive this interpretation without assuming (\ref{analytic_relation}). But we can use similar techniques for the asymptotic analysis.

Using the representation (\ref{function2_rep}) in Lemma \ref{lemma:function2_rep}, we can write:
$$
e^{2sx}g(s,x)=\frac{1}{2\pi}\int_{u=-\infty}^{\infty}\frac{\int_{y=0}^{-4^{1/3}x}e^{-2^{1/3}s(iu+y)}\ai(iu+y)\,dy}{\ai(iu)^2}\,du\,.
$$
So we get:
\begin{equation}
\label{Airy_limit}
\lim_{x\to-\infty}e^{2sx}g(s,x)
=\frac{1}{2\pi}\int_{u=-\infty}^{\infty}\frac{\int_{y=0}^{\infty}e^{-2^{1/3}s(iu+y)}\ai(iu+y)\,dy}{\ai(iu)^2}\,du.
\end{equation}

We now have the following lemma.

\begin{lemma}
\label{lemma:Laplace_transform}
\begin{equation}
\label{Laplace_transforms}
\frac{1}{2\pi}\int_{u=-\infty}^{\infty}\frac{\int_{y=0}^{\infty}e^{-2^{1/3}s(iu+y)}\ai(iu+y)\,dy}{\ai(iu)^2}\,du
=e^{-\tfrac23s^3},\,\forall s\in\R.
\end{equation}
\end{lemma}

\begin{proof}
Note that
\begin{align*}
&\frac{\partial}{\partial s}\int_{y=0}^{\infty}e^{-2^{1/3}s(iu+y)}\ai(iu+y)\,dy=
-2^{1/3}\int_{y=0}^{\infty}e^{-2^{1/3}s(iu+y)}(iu+y)\ai(iu+y)\,dy\\
&=-2^{1/3}\int_{y=0}^{\infty}e^{-2^{1/3}s(iu+y)}\ai''(iu+y)\,dy\\
&=2^{1/3}\ai'(iu)e^{-2^{1/3}isu}+2^{2/3}s\ai(iu)e^{-2^{1/3}isu}-2s^2\int_{y=0}^{\infty}e^{-2^{1/3}s(iu+y)}\ai(iu+y)\,dy.
\end{align*}
Hence, defining the functions $p$ and $q$ by
$$
p(s)=\frac{1}{2\pi}\int_{u=-\infty}^{\infty}\frac{\int_{y=0}^{\infty}e^{-2^{1/3}s(iu+y)}\ai(iu+y)\,dy}{\ai(iu)^2}\,du,\,s\in\R,
$$
and
$$
q(s)=\frac{1}{2\pi}\int_{u=-\infty}^{\infty}\frac{2^{1/3}\ai'(iu)e^{-2^{1/3}isu}+2^{2/3}s\ai(iu)e^{-2^{1/3}isu}}{\ai(iu)^2}\,du,\,s\in\R,
$$
we get:
\begin{align*}
p'(s)=-2s^2p(s)+q(s).
\end{align*}
But we have, by integration by parts,
\begin{align*}
&\int_{u=-\infty}^{\infty}\frac{2^{2/3}s\ai(iu)e^{-2^{1/3}isu}}{\ai(iu)^2}\,du
=\int_{u=-\infty}^{\infty}\frac{2^{2/3}se^{-2^{1/3}isu}}{\ai(iu)}\,du
=-2^{1/3}\int_{u=-\infty}^{\infty}\frac{e^{-2^{1/3}isu}\ai'(iu)}{\ai(iu)^2}\,du,
\end{align*}
implying $q(s)\equiv0$, and hence:
$$
p(s)=ce^{-\tfrac23s^3},
$$
for some constant $c$. For $s=0$, we get:
$$
p(0)=\frac{1}{2\pi}\int_{u=-\infty}^{\infty}\frac{\int_{y=0}^{\infty}\ai(iu+y)\,dy}{\ai(iu)^2}\,du.
$$

By Cauchy's theorem, we now have, for $u\in\R$,
\begin{align*}
\int_{y=0}^{\infty}\ai(iu+y)\,dy=\int_{y=0}^{\infty}\ai(y)\,dy-i\int_{y=0}^{u}\ai(iy)\,dy
=\frac13-i\int_{y=0}^{u}\ai(iy)\,dy,
\end{align*}
where the integrals are interpreted as (directed) Riemann integrals if $u<0$.
Hence, since
$$
\frac{1}{2\pi}\int_{u=-\infty}^{\infty}\frac1{\ai(iu)^2}\,du=1,
$$
(see Appendix D), we get:
$$
p(0)=1/3-\frac{1}{2\pi}\int_{u=-\infty}^{\infty}\frac{i\int_{y=0}^{u}\ai(iv)\,dv}{\ai(iu)^2}\,du=\frac13+\frac23=1,
$$
so $c=1$. The conclusion of the lemma now follows.
\end{proof}

\vspace{0.3cm}
Relation (\ref{Airy_limit}) and Lemma \ref{lemma:Laplace_transform} yield the second part of (\ref{boundary_condition2}). We therefore have:
\begin{align*}
e^{2sx}g(s,x)=e^{-\tfrac23s^3}-\frac{1}{2\pi}\int_{u=-\infty}^{\infty}\frac{\int_{y=-4^{1/3}x}^{\infty}e^{-2^{1/3}s(iu+y)}\ai(iu+y)\,dy}{\ai(iu)^2}\,du,
\end{align*}
implying
\begin{align*}
\frac{\partial}{\partial x}\left\{e^{2sx}g(s,x)\right\}=-\frac{1}{2^{1/3}\pi}\int_{u=-\infty}^{\infty}\frac{e^{-2^{1/3}s(iu-4^{1/3}x)}\ai(iu-4^{1/3}x)\,dy}{\ai(iu)^2}\,du=O\left(e^{-\tfrac23|x|^{3/2}}\right).
\end{align*}
This, in turn, implies that also
\begin{align*}
&e^{-2sx-\tfrac23s^3}-g(s,x)=e^{-2sx}\left\{e^{-\tfrac23s^3}-e^{2sx}g(s,x)\right\}\\
&=\frac{e^{-2sx}}{2\pi}\int_{u=-\infty}^{\infty}\frac{\int_{y=-4^{1/3}x}^{\infty}e^{-2^{1/3}s(iu+y)}\ai(iu+y)\,dy}{\ai(iu)^2}\,du
=O\left(e^{-\tfrac23|x|^{3/2}-2sx}\right)=O\left(e^{-\tfrac23|x|^{3/2}(1+o(1))}\right),
\end{align*}
and hence, for all $s\in\R$,
\begin{align*}
g(s,x)=e^{-\tfrac23s^3-2sx}+o(1),\,x\to-\infty.
\end{align*}

The analysis of $f(s,x)$ is easier and shows, by standard methods, that also, for all $s\in\R$,
$$
e^{2sx}f(s,x)=O\left(e^{-\tfrac23|x|^{3/2}(1+o(1))}\right),\,x\to-\infty,
$$
implying
$$
\lim_{x\to-\infty}f(s,x)=0,\,\forall s\in\R.
$$

Finally, again using the representation (\ref{function2_rep}) in Lemma \ref{lemma:function2_rep}, we get, for each $x<0$:
\begin{align*}
g(s,x)=\frac{e^{-2sx}}{2\pi}\int_{y=0}^{-4^{1/3}x}\left\{\int_{u=-\infty}^{\infty}\frac{e^{-2^{1/3}s(iu+y)}\ai(iu+y)\,du}{\ai(iu)^2}\,du\right\}\,dy.
\end{align*}
We write:
$$
\frac1{2\pi}\int_{u=-\infty}^{\infty}\frac{e^{-2^{1/3}s(iu+y)}\ai(iu+y)\,du}{\ai(iu)^2}\,du
=\frac1{2\pi i}\int_{u=-i\infty}^{i\infty}\frac{e^{-2^{1/3}s(u+y)}\ai(u+y)\,du}{\ai(u)^2}\,du.
$$
Shifting the integration path of the latter integral to the right, letting it pass through the approximate saddle point $2^{-2/3}s^2$ on the real axis, we get for fixed $y>0$ by the asymptotic expansion of the Airy function:
\begin{align*}
\frac1{2\pi i}\int_{u=-i\infty}^{i\infty}\frac{e^{-2^{1/3}s(u+y)}\ai(u+y)}{\ai(u)^2}\,du
= O\left(s^{3/4}e^{-\tfrac23s^3}\right),
\end{align*}
yielding
$$
\lim_{s\to\infty}g(s,x)=0,\,\forall x<0.
$$
The relation
$$
\lim_{s\to\infty}f(s,x)=0,\,\forall x<0,
$$
follows in a similar way.

\section{Appendix D}
\label{section:appendixD}
\setcounter{equation}{0}
We have not been able to find a direct reference for the nice relation
\begin{equation}
\label{AiryInt}
\frac{1}{2\pi}\int_{u=-\infty}^{\infty}\frac1{\ai(iu)^2}\,du=1,
\end{equation}
so we provide a proof here. The indefinite integral of $1/\ai(iu)^2$ appears in \cite{prudnikov:90}, p.\ 30, but further work is needed to obtain (\ref{AiryInt}). Let $h$ be given by 
(\ref{h(u)}). We use
\begin{align*}\label{defhu}
h(u)=\frac1{2\pi\ai(iu)^2}=\frac{d g(u)}{du},\quad g(u)=\frac{e^{-i\pi/6}\ai\bigl(e^{-i\pi/6}u\bigr)}{i\ai(iu)}\,.
\end{align*}
For the proof we use integration by parts and we need the asymptotic behavior of $g(u)$ as $u\to\pm\infty$.

We have as $z\to\infty$
\begin{equation}\label{asyai}
\ai(z)\sim \frac{e^{-\zeta}}{2\sqrt{\pi}z^{\frac14}},\quad  |\text{ph}(z)|<\pi, \quad \zeta=\frac23z^{\frac32}.
\end{equation}
This gives, as $u\to+\infty$,
\begin{align*}\label{asyuplus}
g(u)\sim -ie^{-\frac43u^{\frac32}e^{\frac14\pi i}},
\end{align*}
which is exponentially small and does not give a contribution when integrating by parts.

For $u\to-\infty$ we write
\begin{equation}\label{gminu}
g(-u)=\frac{e^{-i\pi/6}\ai\bigl(e^{5i\pi/6}u\bigr)}{i\ai(-iu)}\,,
\end{equation}
and we have, using (\ref{asyai}),
$$
g(-u)\sim -1,\quad u\to+\infty.
$$
It follows that we need more details of the asymptotic behavior of the Airy function by using the complete expansion
$$
\ai(z)\sim \frac{e^{-\zeta}}{2\sqrt{\pi}z^{\frac14}}\sum_{k=0}^\infty (-1)^k\frac{u_k}{\zeta^k},\quad  |\text{ph}(z)| <\pi, 
$$
with $\zeta$ given in (\ref{asyai}), and the coefficients $u_k$ do not depend on $z$. The first few are
$u_0=1$ and $u_1=\frac{5}{72}$.

For $z=e^{5i\pi/6}u$ and for $z=-iu$ we have the same $\zeta=\frac23u^{3/2} e^{5i\pi/4}=\frac23u^{3/2} e^{-3i\pi/4}$, and we see that the asymptotic series of the Airy functions in (\ref{gminu}) are the same as well. We conclude
$$
g(-u)=-1+O\left(u^{-k}\right),\quad \forall k, \quad u\to+\infty.
$$
This proves the relation in (\ref{AiryInt}).

\section{Appendix E}
\label{section:appendixE}
\setcounter{equation}{0}
\begin{lemma}
\label{lemma:psi_phi}
Let the functions $\f$ and $\psi$ be defined by (\ref{phi}) and (\ref{def_psi}), respectively. Then, for all $t\ge0$,
$$
\psi(t)=\tfrac12\f(-t).
$$
\end{lemma}

\begin{proof}
By (\ref{def_h_x}) we have for $t\ge0$ and $x>0$:
$$
h_{-x}(t)=\frac1{2\pi}\int_{v=-\infty}^{\infty}e^{itv}\frac{\ai(i2^{-1/3}v+4^{1/3}x)}{\ai(i2^{-1/3}v)}\,dv,
$$
and by Lemma \ref{lemma:function2_rep} we have:
$$
g(0,-x)=\frac1{2\pi}\int_{u=-\infty}^{\infty}\frac{\int_{y=0}^{4^{1/3}x}\ai(iu+y)\,dy}{\ai(iu)^2}\,du.
$$
Hence we get:
\begin{align}
\label{def_psi2}
\psi(t)&=\frac1{4\pi^2}\int_{x=0}^{\infty}\left\{\int_{v=-\infty}^{\infty}e^{itv}\frac{\ai(i2^{-1/3}v+4^{1/3}x)}{\ai(i2^{-1/3}v)}\,dv\int_{u=-\infty}^{\infty}\frac{\int_{y=0}^{4^{1/3}x}\ai(iu+y)\,dy}{\ai(iu)^2}\,du\right\}\,dx\nonumber\\
&=\frac1{2^{1/3}\cdot4\pi^2}\int_{x=0}^{\infty}\left\{\int_{v=-\infty}^{\infty}e^{i2^{1/3}tv}\frac{\int_{y=x}^{\infty}\ai(iv+y)\,dy}{\ai(iv)}\,dv\int_{u=-\infty}^{\infty}\frac{\ai(iu+x)}{\ai(iu)^2}\,du\right\}\,dx\nonumber\\
&=\frac1{2^{1/3}\cdot4\pi^2}\int_{x=0}^{\infty}\left\{\int_{v=-\infty}^{\infty}e^{i2^{1/3}t(v-u)}\frac{\int_{y=x}^{\infty}\ai(iv+y)\,dy}{\ai(iv)}\,dv\int_{u=-\infty}^{\infty}\frac{e^{i2^{1/3}tu}\ai(iu+x)}{\ai(iu)^2}\,du\right\}\,dx.
\end{align}

Let $t>0$. For convenience of notation, we consider $\psi(2^{-1/3}t)$ instead of $\psi(t)$ (this makes the first factor in the inner integral in the expression on the right-hand side $\exp\{itv\}$ instead of $\exp\{i2^{1/3}tv\}$). Introducing an extra parameter $\e>0$, we first study the behavior of:
\begin{align*}
&\int_{y=x\e}^{\infty}\int_{v=-\infty}^{\infty}e^{it(v-u)}\frac{\ai(iv+y)}{\ai(iv)}\,dv\,dy
=\frac1{i}\int_{y=x\e}^{\infty}\int_{v=-i\infty}^{i\infty}e^{tv}\frac{\ai(v+u+y)}{\ai(v+u)}\,dv\,dy.
\end{align*}
We move the integration path of the inner integral to a path, parallel to the imaginary axis, and passing through the (approximate saddle) point $\tfrac14(y/t)^2$ on the real axis.
This yields, for $y>0$ and $|v|\to\infty$,
\begin{align*}
&\frac1i\int_{v=-i\infty}^{i\infty}e^{tv}\frac{\ai(v+u+y)}{\ai(v+u)}\,dv\sim
\exp\left\{-\frac{y^2}{4t}\right\}\frac{y\sqrt{\pi}}{t^{3/2}}\,.
\end{align*}
Integration of the expression on the right-hand side on the interval $[x\e,\infty)$ gives:
\begin{align*}
\int_{y=x\e}^{\infty}\exp\left\{-\frac{y^2}{4t\e^{2}}\right\}\frac{y\sqrt{\pi}}{t^{3/2}}\,dy=\frac{2e^{-\tfrac14\e^2x^2/t}\sqrt{\pi}}{\sqrt{t}}\,.
\end{align*}
So we find:
\begin{align*}
\psi(2^{-1/3}t)
&=\lim_{\e\downarrow0}\frac{\e}{2^{1/3}\cdot4\pi^2}\int_{x=0}^{\infty}\int_{y=\e x}^{\infty}\left\{\int_{v=-\infty}^{\infty}e^{it(v-u)}\frac{\ai(iv+y)}{\ai(iv)}\,dv\,dy\int_{u=-\infty}^{\infty}\frac{e^{i2^{1/3}tu}\ai(iu+\e x)}{\ai(iu)^2}\,du\right\}\,dx\\
&=\lim_{\e\downarrow0}\frac{\e}{2^{1/3}\cdot4\pi^2}\int_{x=0}^{\infty}\left\{\sqrt{\pi}\int_0^{\infty}\frac{e^{-\tfrac14\e^2x^2/t}}{\sqrt{t}}\int_{u=-\infty}^{\infty}\frac{e^{itu}\ai(iu+x\e)}{\ai(iu)^2}\,du\right\}\,dx\\
&=\frac{1}{2^{4/3}\pi}\int_{u=-\infty}^{\infty}\frac{e^{itu}}{\ai(iu)}\,du
=\frac12\frac{1}{2^{1/3}\pi}\int_{u=-\infty}^{\infty}\frac{e^{itu}}{\ai(iu)}\,du,
\end{align*}
noting that the contribution to the integrals on the right-hand side of the first equality over the region $\{v:|v|\le M\}$ tends to zero for each $M>0$, which implies that we can shift to the asymptotic expansion of the Airy function $\ai$ in the first inner integral.
Hence:
$$
\psi(t)=\frac12\frac{1}{2^{1/3}\pi}\int_{u=-\infty}^{\infty}\frac{e^{i2^{1/3}tu}}{\ai(iu)}\,du=\frac12\frac{1}{4^{1/3}\pi}\int_{u=-\infty}^{\infty}\frac{e^{itu}}{\ai(i2^{-1/3}u)}\,du=\tfrac12\f(-t).
$$
At $t=0$ we get
$$
\psi(0)=\lim_{t\downarrow0}\psi(t)
=\tfrac12\lim_{t\downarrow0}\f(-t)=\tfrac12\f(0),
$$
which shows that the relation also holds at $t=0$.
\end{proof}

\bibliographystyle{amsplain}
\bibliography{cupbook}

\end{document}